\documentclass[a4paper,12pt,twoside]{amsart}
\usepackage[english]{babel}
\usepackage{amsfonts}
\usepackage{amsthm}
\usepackage{amssymb, amsmath, eucal, enumerate}
\usepackage[all]{xy}
\usepackage[left=3cm,right=3cm,top=3cm,bottom=3cm]{geometry}

\usepackage[svgnames,hyperref]{xcolor}
\newcommand{\mycolor}{Navy}
\usepackage[pdfauthor={Hoang-Son Do and Giang Le}, colorlinks, linktocpage, citecolor = \mycolor, linkcolor = \mycolor, urlcolor = \mycolor]{hyperref}

\newtheorem{theorem}{Theorem}[section]

\newtheorem{corollary}[theorem]{Corollary}
\newtheorem{lemma}[theorem]{Lemma}

\newcommand{\supp}{{\rm supp }}

\newcommand{\C}{\mathbb{C}}
\newcommand{\R}{\mathbb{R}}

\newcommand{\Z}{\mathbb{Z}}

\newcommand{\E}{\mathcal{E}}

\newcommand{\psh}{\text{PSH}}
\newcommand{\pshn}{\text{PSH}^-}
\newcommand{\capK}{\rm cap}

\begin{document}
 \title[Integration by parts]
 {Integration by parts for plurisubharmonic functions} 
\setcounter{tocdepth}{1}
\author{Hoang-Son Do}
\address{Institute of Mathematics, Vietnam Academy of Science and Technology, 18 Hoang Quoc Viet road, Nghia Do ward, 11307 Hanoi, Vietnam}
\email{dhson@math.ac.vn, hoangson.do.vn@gmail.com}
\author{Giang Le}
\address{Department of Mathematics, Hanoi National University of Education, 136 Xuan Thuy road, Cau Giay ward, 11314  Hanoi, Vietnam}
\email{legiang@hnue.edu.vn}
\date{\today }

\begin{abstract}
In this paper, we provide an integration by parts formula for plurisubharmonic functions on a hyperconvex domain that are
bounded outside a compact set. This extends a previous result of Urban Cegrell. 
\end{abstract}


\maketitle

\tableofcontents
\section{Introduction}
Let $\Omega$ be a bounded hyperconvex domain in $\mathbb{C}^n$. Cegrell \cite{Ceg04} has proved the following estimate:
\begin{theorem}[\cite{Ceg04}, Theorem 3.2]
	Let $u$ and $v$ be negative plurisubharmonic functions on $\Omega$ such that 
	$$\lim_{z\to\partial\Omega}u(z)=0.$$
	Suppose that $T$ is a positive and closed current of bidegree $(n-1, n-1)$ on $\Omega$. Then $dd^cu\wedge T$
	is a well-defined positive measure on $\Omega$. Furthermore, if 
	$$\int_{\Omega}vdd^cu\wedge T>-\infty,$$
	then $dd^cv\wedge T$ is also a well-defined positive measure on $\Omega$ and
	$$\int_{\Omega}vdd^cu\wedge T\leq \int_{\Omega}udd^cv\wedge T.$$
\end{theorem}
As a consequence, if $\lim_{z\to\partial\Omega}u(z)=\lim_{z\to\partial\Omega}v(z)=0$ and $T=dd^cw_1\wedge...\wedge dd^c w_{n-1}$
with $w_1,...,w_{n-1}\in\psh\cap L^{\infty}(\Omega)$ then
	$$\int_{\Omega}vdd^cu\wedge T=\int_{\Omega}udd^cv\wedge T.$$
	This integration by parts formula is an essential tool in \cite{Ceg04}, especially in the study of the convergence of (mixed)
	Monge-Amp\`ere operators.
	
	 Our goal is to extend the above result to the case where the boundary conditions of $u$ and $v$ are not necessarily equal to zero.
	 The main result of this paper is as follows:
\begin{theorem}\label{main}
		Let $u, U, v$ and $V$ be negative plurisubharmonic functions on $\Omega$ satisfying the following condition:
		\begin{itemize}
			\item [(i)] $U$ and $V$ are bounded;
			\item [(ii)] $u\leq U$ and $v\leq V$;
			\item [(iii)] for every $z_0\in\partial\Omega$,
			$$\lim_{z\to z_0}(u-U)(z)=\lim_{z\to z_0}(v-V)(z)=0.$$
		\end{itemize}
	Suppose that $w_1,\dots, w_{n-1}\in\mathcal{E}(\Omega)$ and denote
	$T=dd^cw_1\wedge\dots\wedge dd^c w_{n-1}$.
	Then
	\begin{center}
		\begin{align*}
				\int_{\Omega}(u-U) dd^cv\wedge T
			+\int_{\Omega}(v-V) dd^cU\wedge T
			&=\int_{\Omega}(v-V)dd^cu\wedge T\\
			&+\int_{\Omega}(u-U) dd^cV\wedge T.
		\end{align*}
	\end{center}	
\end{theorem}
Here, $\mathcal{E}(\Omega)$ is a class introduced by Cegrell \cite{Ceg04} which coincides with the set of negative plurisubharmonic
 functions belonging to the domain of definition of the  Monge-Amp\`ere operator. We briefly recall some important properties of this class in the Preliminaries.
\section{Preliminaries}
 In the classical sense, if $u$ is a smooth plurisubharmonic function on a domain of $\C^n$ then its
Monge-Amp\`ere operator $(dd^cu)^n$ is equal to $c_n\det (u_{\alpha\bar{\beta}})dV$, 
where $dV$ is the standard volume form and $c_n>0$ depends only on $n$. Building on the foundational work of Bedford and Taylor \cite{BT76, BT82}, one can define the Monge-Amp\`ere operator for bounded plurisubharmonic functions so that $(dd^cu)^n$ is a positive measure. Moreover, it satisfies the following property: if $u_j$ is a sequence of bounded plurisubharmonic 
functions decreasing to $u$ on an open set $U\subset\Omega$ then $(dd^cu_j)^n$ converges weakly to  $(dd^c u)^n$ on $U$.
This leads to the following notion: the Monge-Amp\`ere operator of a plurisubharmonic function $u$ is said to be well defined if there exists a positive measure $(dd^cu)^n$ satisfying the above convergence property.
 We denote by $\mathcal{D}(\Omega)$ the class of plurisubharmonic functions for which the Monge-Amp\`ere operator is well defined.
 It was shown that $PSH\cap L^{\infty}_{loc}\varsubsetneq \mathcal{D}
 \varsubsetneq PSH$. Characterizations of $\mathcal{D}$ were given by Cegrell \cite{Ceg04} (for the case of hyperconvex domains)
 and B\l ocki \cite{Blo06}.
 
In \cite{Ceg98, Ceg04}, Cegrell has introduced the following classes of plurisubharmonic functions on hyperconvex domains:
\begin{center}
	$\mathcal{E}_0(\Omega)=\{u\in PSH^-(\Omega)\cap L^\infty (\Omega): \lim_{z\to\partial\Omega}u(z)=0, \int_\Omega (dd^cu)^n<\infty\},$\\ 
	$\mathcal{F}(\Omega)=\{u\in PSH^-(\Omega): \exists \, \{u_j\}\subset \E_0(\Omega), \; u_j\searrow u, \, \sup_{j}\int_{\Omega} (dd^cu_j)^n<\infty\}$ ,\\
	$\mathcal{E}(\Omega)
	=\{ u\in PSH^-(\Omega):\forall K\Subset\Omega, \exists u_K\in\mathcal{F}(\Omega)
	\mbox{ such that } u_K=u \mbox{ on } K\}.$
\end{center}
By \cite{Ceg04, Blo06}, $\mathcal{E}(\Omega)$ coincides with the class of negative plurisubharmonic functions for which the Monge-Amp\`ere operator is well defined. Moreover,
\begin{itemize}
	\item if $u, v\in\mathcal{E}(\Omega)$ then $u+v\in \mathcal{E}(\Omega)$;
	\item  if $u\in \mathcal{E}(\Omega)$ then $w\in \mathcal{E}(\Omega)$ for every $w\in\pshn (\Omega)$ with $w\geq u$.
\end{itemize}
The following lemma provides a uniform estimate for the total mass of the Monge-Amp\`ere measure on compact subsets of $\Omega$.
\begin{lemma}\label{lem1.pre}
	Let $u\in\mathcal{E}(\Omega)$.
	 Then, for every compact subset $K$ of $\Omega$, there exists a constant $C>0$ such that
	 $$\int_K(dd^cw)^n\leq C,$$
	 for every $w\in\mathcal{E} (\Omega)$ with $w\geq u$.
\end{lemma}
\begin{proof}
	By the definition of the class $\mathcal{E}$, there exists a function $\tilde{u}\in\mathcal{F}(\Omega)$ such that
	$u=\tilde{u}$ on a neighborhood $W$ of $K$. Note that $\int_{\Omega}(dd^c\tilde{u})^n<\infty$ (by \cite[Proposition 5.1]{Ceg04}).
	
	Put $\tilde{w}=\max\{w, \tilde{u}\}$. Since $w\geq u=\tilde{u}$ on $W$, we have $w=\tilde{w}$ on $W$.
	Applying \cite[Lemma 3.3]{ACCP09} to the case where $H=0$ and $\varphi=-1$ (use the fact $\mathcal{F}\subset\mathcal{N}$), we have
	$$\int_K(dd^cw)^n=\int_K(dd^c\tilde{w})^n\leq \int_{\Omega}(dd^c\tilde{w})^n\leq \int_{\Omega}(dd^c\tilde{u})^n:=C<\infty.$$
	The proof is complete.	
\end{proof}
In the following, we recall some convergence properties of the Monge-Amp\`ere operator on $\mathcal{E}(\Omega)$.
\begin{theorem}[\cite{P10}, Theorem 3.1]\label{the.P10}
	Let $u_j, v_j, w\in\mathcal{E}(\Omega)$ be such that $u_j, v_j\geq w$ for every $j\in\mathbb{Z}^+$. Assume that $|u_j-v_j|$ 
	converges to $0$ in capacity as $j\rightarrow\infty$. Then, $\varphi ((dd^c u_j)^n-(dd^cv_j)^n)$ converges to $0$ in the weak-topology
	of measures for every $\varphi\in\psh \cap L^{\infty}(\Omega)$.
\end{theorem}
Here, we say that $|u_j-v_j|$ converges to $0$ in capacity if 
$$\lim_{j\to\infty}\capK(\{|u_j-v_j|>\epsilon\}, \Omega)=0,$$
for every $\epsilon>0$, where
$$\capK (E, \Omega)=\left\{\int_{E}(dd^cw)^n: w\in\psh(\Omega), 0\leq w\leq 1\right\},$$
for every Borel subset $E$ of $\Omega$. It is classical that if a sequence $(u_j)$ of plurisubharmonic functions converges monotonically
to a plurisubharmonic function $u$, then $|u_j-u|$ converges to $0$ in capacity (see, for example, \cite{Kli91}).

The above theorem can be slightly extended as follows
\begin{theorem}\label{the.P10extend}
	Let $u_j^p, v_j^p\in\mathcal{E}(\Omega)$, $1\leq p\leq n$, such that
	\begin{itemize}
		\item [(i)] there exists $w\in\mathcal{E}(\Omega)$ such that $u_j^p, v_j^p\geq w$ for every $1\leq p\leq n$ and
		for all $j\in\mathbb{Z}^+$;
		\item [(ii)] $|u_j-v_j|$ 
		converges to $0$ in capacity as $j\rightarrow\infty$ for every $1\leq p\leq n$. Assume that
		$\varphi$ is a bounded plurisubharmonic function on $\Omega$. Then, 
		$$\lim_{j\to\infty}\int_{\Omega}\chi\,\varphi\,((dd^c u_j^1\wedge...\wedge dd^c u_j^n)-(dd^c v_j^1\wedge...\wedge dd^c v_j^n))
		=0,$$
		for every test function $\chi\in C_c(\Omega)$.
	\end{itemize}
\end{theorem}
\begin{proof}
	By decomposing $\chi=\max\{\chi,0\}-\max\{-\chi, 0\}$, it suffices to consider the case where $\chi\geq 0$. 
	
	For every multi-index $\alpha=(\alpha_1,...\alpha_n)\in\mathbb{N}^n$ with $|\alpha|=\sum_{p=1}^n\alpha_p=n$, we denote
	$$a_{\alpha, j}=\int_{\Omega}\chi\,\varphi\,((dd^c u_j^1)^{\alpha_1}\wedge...\wedge (dd^c u_j^n)^{\alpha_n}
	-(dd^c v_j^1)^{\alpha_1}\wedge...\wedge (dd^c v_j^n)^{\alpha_n}),$$
	for every $j\in\mathbb{Z}^+$.

	Put $K=\supp \chi$ and $M=\sup_{\Omega}|\varphi|$. By Lemma \ref{lem1.pre}, there exists $C_1>0$ such that
	$$\int_{\Omega}\chi\,|\varphi|\,(dd^c(u_j^1+...+u_j^n))^n\leq M\int_K(dd^c(u_j^1+...+u_j^n))^n\leq C_1,$$
	and
		$$\int_{\Omega}\chi\,|\varphi|\,(dd^c(v_j^1+...+v_j^n))^n\leq M\int_K(dd^c(v_j^1+...+v_j^n))^n\leq C_1,$$
		for every $j\in\mathbb{Z}^+$. Then, we have
	\begin{equation}\label{eq1.P10extend}
	|a_{\alpha, j}|\leq 2C_1,	
	\end{equation}
	for every $j, \alpha$.
	
	For $t=(t_1,...,t_n)\in (\mathbb{R}^+)^n$, we define
	$$P_j(t)=\int_{\Omega}\chi\,\varphi\,((dd^c(t_1u_j^1+...+t_nu_j^n))^n-(dd^c(t_1v_j^1+...+t_nv_j^n))^n)
		=n!\sum_{|\alpha|=n}\frac{a_{\alpha, j}}{\alpha!}t^{\alpha},$$
	for every $j\in\mathbb{Z}^+$.
	
	By Theorem \ref{the.P10}, we have $P_j$ converges pointwise to $0$ as $j\rightarrow\infty$. Moreover, it follows from
	\eqref{eq1.P10extend} that 
	$$\sup_{[0, 1]^n}|\nabla P_j|\leq C_2,$$
	for every $j\in\mathbb{Z}^+$, where $C_2>0$ is a constant. Then, by the Arzel\`a-Ascoli theorem, we have
	$P_j$ converges uniformly to $0$ on $[0, 1]^n$. Since the space of homogeneous polynomials of degree $n$ is finite-dimensional, any two norms on this space are equivalent. Hence, it follows that
	$$\lim_{j\to\infty}a_{\alpha, j}=0,$$
	for every $\alpha$. In particular
		$$\lim_{j\to\infty}\int_{\Omega}\chi\,\varphi\,((dd^c u_j^1\wedge...\wedge dd^c u_j^n)-(dd^c v_j^1\wedge...\wedge dd^c v_j^n))
	=\lim_{j\to\infty}a_{(1,...,1), j}=0.$$
	The proof is complete.
\end{proof}
\begin{corollary}\label{cor.P10}
	Let $w_j^p, w^p\in\mathcal{E}(\Omega)$, $1\leq p\leq n$, such that $w_j^p$ converges monotonically to $w^p$ as $j\rightarrow\infty$.  Suppose that $u$ and $v$ are negative plurisubharmonic functions in $\Omega$ satisfying $u\leq v$. Assume in addition that
	$v$ is bounded, so that $u-v$ is well-defined on the whole $\Omega$. Then
	$$\limsup_{j\to\infty}\int_{\Omega}(u-v)\,dd^c w_j^1\wedge...\wedge dd^c w_j^n
	\leq \int_{\Omega}(u-v)\,dd^c w^1\wedge...\wedge dd^c w^n.$$
	Moreover, if $u$ is also bounded and $u=v$ outside a compact subset of $\Omega$ then
	$$\lim_{j\to\infty}\int_{\Omega}(u-v)\,dd^c w_j^1\wedge...\wedge dd^c w_j^n
	= \int_{\Omega}(u-v)\,dd^c w^1\wedge...\wedge dd^c w^n.$$ 
\end{corollary}
\begin{proof}
	Let $(\Omega_k)$ be an exhaustion of $\Omega$ by relatively compact open subsets. For each $k\in\mathbb{Z}^+$, let 
	$\chi_k\in C_c(\Omega)$ be a test function satisfying $0\leq\chi_k\leq 1$ and $\Omega_k\subset\supp \chi_k$. By Theorem
	\ref{the.P10extend}, we have
	\begin{equation}\label{eq1.cor2}
		\lim_{j\to\infty}\int_{\Omega}(u_k-v)\chi_k\,dd^c w_j^1\wedge...\wedge dd^c w_j^n
		= \int_{\Omega}(u_k-v)\chi_k\,dd^c w^1\wedge...\wedge dd^c w^n,
	\end{equation}
	for every $k\in\mathbb{Z}^+$, where $u_k=\max\{u, -k\}$.
	
	Assume $k>\sup_{\Omega}|v|$. We have
	$$u-v\leq (u_k-v)\leq (u_k-v)\chi_k\leq (u_k-v)\mathbf{1}_{\Omega_k}\leq 0.$$
	Combining this with \eqref{eq1.cor2}, we get
		$$\limsup_{j\to\infty}\int_{\Omega}(u-v)\,dd^c w_j^1\wedge...\wedge dd^c w_j^n
	\leq\int_{\Omega_k}(u_k-v)\,dd^c w^1\wedge...\wedge dd^c w^n,$$
	for every $k>\sup_{\Omega}|v|$. Letting $k\rightarrow\infty$ and using the Monotone Convergence Theorem, we obtain
		$$\limsup_{j\to\infty}\int_{\Omega}(u-v)\,dd^c w_j^1\wedge...\wedge dd^c w_j^n
	\leq\int_{\Omega}(u-v)\,dd^c w^1\wedge...\wedge dd^c w^n.$$
	Now, we consider the case where $u$ is bounded and $u=v$ outside a compact subset of $\Omega$. For $k\gg 1$, we have 
	 $u_k=u$ on $\Omega$ and $u-v=0$ outside the set $\{\chi_k=1\}$. Then, it follows from \eqref{eq1.cor2} that
	 	$$\lim_{j\to\infty}\int_{\Omega}(u-v)\,dd^c w_j^1\wedge...\wedge dd^c w_j^n
	 = \int_{\Omega}(u-v)\,dd^c w^1\wedge...\wedge dd^c w^n.$$
	The proof is completed.
\end{proof}

\section{Proof of the main theorem}
In order to prove Theorem \ref{main}, we need the following lemma:
\begin{lemma}\label{lem-new.Int.by.parts }
	Let $\Omega\subset\mathbb{C}^n$ be a domain. Let $u_1, u_2, v_1, v_2$ be bounded plurisubharmonic functions in  $\Omega$ such that $u_1=u_2$ and $v_1=v_2$ 
	outside a compact subset $K$ of $\Omega$. Then
	$$\int_{\Omega}(u_1-u_2)dd^c(v_1-v_2)\wedge T=\int_{\Omega}(v_1-v_2)dd^c(u_1-u_2)\wedge T,$$
	where $T=dd^cw_1\wedge\dots\wedge dd^c w_{n-1}$, and $w_1,\dots, w_{n-1}:\Omega\longrightarrow \R$
	are bounded plurisubharmonic functions.
\end{lemma}
\begin{proof} We use the same method as in the proof of \cite[Theorem 3.1]{Coman97}.
	Let $\chi\in C_c^{\infty}(\Omega)$ be a test function such that $0\leq \chi\leq 1$ in $\Omega$ and $\chi\equiv 1$ in a neighborhood
	of	$K$. Since $u_1=u_2$ and $v_1=v_2$ outside $K$, we have
	$$\int_{\Omega}(1-\chi)(u_1-u_2)dd^c(v_1-v_2)\wedge T=\int_{\Omega}(1-\chi)(v_1-v_2)dd^c(u_1-u_2)\wedge T.$$
	Then it remains to show that
	$$\int_{\Omega}\chi\,(u_1-u_2)dd^c(v_1-v_2)\wedge T=\int_{\Omega}\chi\,(v_1-v_2)dd^c(u_1-u_2)\wedge T.$$
	Let $V$ be an open relative compact subset of $\Omega$ such that $\supp \chi\subset V$. By using convolution,
	we can construct sequences $\{u_1^j\}$, $\{u_2^j\}$, $\{v_1^j\}$, $\{v_2^j\}$ and $\{w_k^j\}$, $k=1,...,n-1$, of smooth
	plurisubharmonic functions in $V$ satisfying the following conditions:
	\begin{itemize}
		\item $u_l^j$ is decreasing to $u_l$ as $j\rightarrow\infty$ for $l=1, 2$;
		\item $v_l^j$ is decreasing to $v_l$ as $j\rightarrow\infty$ for $l=1, 2$;
		\item $u_1^j=u_2^j$ and $v_1^j=v_2^j$ outside a compact subset $\tilde{K}$ of $V$ for every $j\in\Z^+$;
		\item $w_k^j$ is decreasing to $w_k$ as $j\rightarrow\infty$ for every $k=1, 2,..., n-1$.
	\end{itemize}
	Set
	$$T_j=dd^cw_1^j\wedge\dots\wedge dd^c w_{n-1}^j.$$
	By Stokes' theorem, we have
	\begin{equation}\label{eq1lem4.1}
		\int_{V}\chi\,(u_1^j-u_2^j)dd^c(v_1^j-v_2^j)\wedge T_j=\int_{V}\chi\,(v_1^j-v_2^j)dd^c(u_1^j-u_2^j)\wedge T_j,
	\end{equation}
	for every $j\in\Z^+$. Moreover, by \cite[Theorem 3.2]{BT87}, $(u_1^j-u_2^j)dd^c(v_1^j-v_2^j)\wedge T_j$ converges to
	$(u_1-u_2)dd^c(v_1-v_2)\wedge T$ and $(v_1^j-v_2^j)dd^c(u_1^j-u_2^j)\wedge T_j$ converges to
	$(v_1-v_2)dd^c(u_1-u_2)\wedge T$ in the sense of currents of order $0$ on $V$ as $j\rightarrow\infty$. In particular, we have
	\begin{equation}\label{eq2lem4.1}
		\lim_{j\to\infty}\int_{V}\chi\,(u_1^j-u_2^j)dd^c(v_1^j-v_2^j)\wedge T_j=
		\int_V\chi\,(u_1-u_2)dd^c(v_1-v_2)\wedge T,
	\end{equation}
	and
	\begin{equation}\label{eq3lem4.1}
		\lim_{j\to\infty}\int_{V}\chi\,(v_1^j-v_2^j)dd^c(u_1^j-u_2^j)\wedge T_j=\int_V\chi\,(v_1-v_2)dd^c(u_1-u_2)\wedge T.
	\end{equation}
	Combining \eqref{eq1lem4.1}, \eqref{eq2lem4.1} and \eqref{eq3lem4.1}, we get
	$$\int_V\chi\,(u_1-u_2)dd^c(v_1-v_2)\wedge T=\int_V\chi\,(v_1-v_2)dd^c(u_1-u_2)\wedge T.$$
	Since $\supp\chi\subset V$, it follows that
	$$\int_{\Omega}\chi\,(u_1-u_2)dd^c(v_1-v_2)\wedge T=\int_{\Omega}\chi\,(v_1-v_2)dd^c(u_1-u_2)\wedge T.$$
	This completes the proof.
\end{proof}
We now use the above lemma to prove Theorem \ref{main} for the case where $u$ and $v$ are bounded:
\begin{theorem}\label{the.main-bounded}
	Let $u, U, v$ and $V$ be negative plurisubharmonic functions on $\Omega$ satisfying the following condition:
	\begin{itemize}
		\item [(i)] $u, v, U$ and $V$ are bounded;
		\item [(ii)] $u\leq U$ and $v\leq V$;
		\item [(iii)] for every $z_0\in\partial\Omega$,
		$$\lim_{z\to z_0}(u-U)(z)=\lim_{z\to z_0}(v-V)(z)=0.$$
	\end{itemize}
	Suppose that $w_1,\dots, w_{n-1}\in\mathcal{E}(\Omega)$ and denote
	$T=dd^cw_1\wedge\dots\wedge dd^c w_{n-1}$.
	Then
	\begin{center}
		\begin{align*}
			\int_{\Omega}(u-U) dd^cv\wedge T
			+\int_{\Omega}(v-V) dd^cU\wedge T
			&=\int_{\Omega}(v-V)dd^cu\wedge T\\
			&+\int_{\Omega}(u-U) dd^cV\wedge T.
		\end{align*}
	\end{center}	
\end{theorem}
\begin{proof} 
	Since $(u, U)$ and $(v, V)$ play the same role, it suffices to prove that
	\begin{center}
	\begin{align*}
		\int_{\Omega}(u-U) dd^cv\wedge T
		+\int_{\Omega}(v-V) dd^cU\wedge T
		&\geq\int_{\Omega}(v-V)dd^cu\wedge T\\
		&+\int_{\Omega}(u-U) dd^cV\wedge T.
	\end{align*}
\end{center}
	For  $m\in\mathbb{Z}^+$, we denote
		$$U_m=\max\{u, U-2^{-m}\},\qquad V_m=\max\{v, V-2^{-m}\},$$
		and
	$$w_{p, m}=\max\{w_p, -m\}, \, 1\leq p\leq n-1.$$
	 Since $(u-U)(z)$ and $(v-V)$ tend to $0$ as $z\rightarrow\partial\Omega$, we have $u=U_m$ and $v=V_m$ outside a compact
	subset of $\Omega$ for each $m\in\mathbb{Z}^+$. Using Lemma \ref{lem-new.Int.by.parts }, we get
	$$\int_{\Omega}(u-U_k)dd^c(v-V_l)\wedge T_m=\int_{\Omega}(v-V_l)dd^c(u-U_k)\wedge T_m,$$
	for every $k, l, m\in\mathbb{Z}^+$, where $T_m=dd^cw_{1, m}\wedge...\wedge dd^cw_{n-1, m}$.
	
	Letting $m\rightarrow\infty$ and applying Corollary \ref{cor.P10}, we have
		$$\int_{\Omega}(u-U_k)dd^c(v-V_l)\wedge T=\int_{\Omega}(v-V_l)dd^c(u-U_k)\wedge T,$$
	for every $k, l\in\mathbb{Z}^+$. This equality can be rewritten as
	$$\int_{\Omega}(u-U_k)dd^cv\wedge T+\int_{\Omega}(v-V_l)dd^cU_k\wedge T
	=\int_{\Omega}(v-V_l)dd^cu\wedge T+\int_{\Omega}(u-U_k)dd^c V_l\wedge T.$$
	Let $l\rightarrow\infty$. By using Corollary \ref{cor.P10} again and using the Monotone Convergence Theorem, it follows that
	\begin{equation}\label{eq1.main}
		\int_{\Omega}(u-U_k)dd^cv\wedge T+\int_{\Omega}(v-V)dd^cU_k\wedge T
		=\int_{\Omega}(v-V)dd^cu\wedge T+\int_{\Omega}(u-U_k)dd^c V\wedge T.
	\end{equation}
	By the Monotone Convergence Theorem, we have
	\begin{equation}\label{eq2.main}
		\lim_{k\rightarrow\infty}\int_{\Omega}(u-U_k)dd^cv\wedge T=\int_{\Omega}(u-U)dd^cv\wedge T,
	\end{equation}
	and
	\begin{equation}\label{eq4.main}
			\lim_{k\rightarrow\infty}\int_{\Omega}(u-U_k)dd^cV\wedge T=\int_{\Omega}(u-U)dd^cV\wedge T.
	\end{equation}
	By Corollary \ref{cor.P10}, we have
	\begin{equation}\label{eq5.main}
			\limsup_{k\rightarrow\infty}\int_{\Omega}(v-V)dd^cU_k\wedge T\leq \int_{\Omega}(v-V)dd^cU\wedge T.
	\end{equation}
	Combining \eqref{eq1.main}, \eqref{eq2.main}, \eqref{eq4.main} and \eqref{eq5.main}, we obtain
	\begin{center}
	\begin{align*}
		\int_{\Omega}(u-U) dd^cv\wedge T
		+\int_{\Omega}(v-V) dd^cU\wedge T
		&\geq\int_{\Omega}(v-V)dd^cu\wedge T\\
		&+\int_{\Omega}(u-U) dd^cV\wedge T,
	\end{align*}
\end{center}
	as desired.
\end{proof}
\begin{proof}[End of the proof of Theorem \ref{main}]
		Since $(u, U)$ and $(v, V)$ play the same role, it suffices to prove that
	\begin{center}
		\begin{align*}
			\int_{\Omega}(u-U) dd^cv\wedge T
			+\int_{\Omega}(v-V) dd^cU\wedge T
			&\geq\int_{\Omega}(v-V)dd^cu\wedge T\\
			&+\int_{\Omega}(u-U) dd^cV\wedge T.
		\end{align*}
	\end{center}
	For  $m\in\mathbb{Z}^+$, we denote
	$$u_m=\max\{u, U-2^{m}\}\quad\mbox{and}\quad v_m=\max\{v, V-2^{m}\}.$$
	By Theorem \ref{the.main-bounded}, we have
			$$\int_{\Omega}(u_k-U) dd^cv_l\wedge T
		+\int_{\Omega}(v_l-V) dd^cU\wedge T
		=\int_{\Omega}(v_l-V)dd^cu_k\wedge T
		+\int_{\Omega}(u_k-U) dd^cV\wedge T,$$
	for every $k, l\in\mathbb{Z}^+$. Then
	\begin{equation}\label{eq1.proofmain}
		\int_{\Omega}(u_m-U) dd^cv_l\wedge T
		+\int_{\Omega}(v_l-V) dd^cU\wedge T
		\geq \int_{\Omega}(v_l-V)dd^cu_k\wedge T
		+\int_{\Omega}(u-U) dd^cV\wedge T,
	\end{equation}
	for every $k, l, m\in\mathbb{Z}^+$ with $m\leq k$.
	Since $(u-U)(z)\rightarrow 0$ as $z\rightarrow\partial\Omega$,
	there exists a compact subset $K$ of $\Omega$ such that $u_k=u$ on $\Omega\setminus K$ for every $k\in\mathbb{Z}^+$. 
	 Let 
	$\chi\in C_c(\Omega)$ be a test function satisfying $0\leq\chi\leq 1$ and $K\subset\supp \chi$. For $k\geq  1$, since
	$1-\chi=0$ on $K$ and $u=u_k$ on $\Omega\setminus K$, we have
	\begin{equation}\label{eq2.proofmain}
		\int_{\Omega}(v_l-V)(1-\chi) dd^cu_k\wedge T=\int_{\Omega}(v_l-V)(1-\chi) dd^cu\wedge T.
	\end{equation}
	Moreover, by Theorem \ref{the.P10extend}, we have
	\begin{equation}\label{eq3.proofmain}
		\lim_{k\to\infty}\int_{\Omega}(v_l-V)\chi\, dd^cu_k\wedge T=\int_{\Omega}(v_l-V)\chi\, dd^cu\wedge T.
	\end{equation}
	By \eqref{eq2.proofmain} and \eqref{eq3.proofmain}, we get
	\begin{equation}\label{eq4.proofmain}
		\lim_{k\to\infty}\int_{\Omega}(v_l-V)\, dd^cu_k\wedge T=\int_{\Omega}(v_l-V)\, dd^cu\wedge T.
	\end{equation}
	Combining \eqref{eq1.proofmain} and \eqref{eq4.proofmain}, we get
		$$\int_{\Omega}(u_m-U) dd^cv_l\wedge T
		+\int_{\Omega}(v_l-V) dd^cU\wedge T
		\geq \int_{\Omega}(v_l-V)dd^cu\wedge T
		+\int_{\Omega}(u-U) dd^cV\wedge T,$$
	for every $m, l\in\mathbb{Z}^+$. Let $l\rightarrow\infty$. By Corollary \ref{cor.P10} and the Monotone Convergence Theorem, we have
		$$\int_{\Omega}(u_m-U) dd^cv\wedge T
	+\int_{\Omega}(v-V) dd^cU\wedge T
	\geq \int_{\Omega}(v-V)dd^cu\wedge T
	+\int_{\Omega}(u-U) dd^cV\wedge T.$$
	Letting $m\rightarrow\infty$, we obtain
		$$\int_{\Omega}(u-U) dd^cv\wedge T
	+\int_{\Omega}(v-V) dd^cU\wedge T
	\geq \int_{\Omega}(v-V)dd^cu\wedge T
	+\int_{\Omega}(u-U) dd^cV\wedge T.$$
	The proof is complete.
\end{proof}
\noindent
\textbf{Data availability.} Data sharing not applicable to this article as no datasets were generated or analysed during the current study.
\\

\noindent
\textbf{\Large Declarations}
\\

\noindent
\textbf{Conflict of interest.} The authors declare that there is no conflict of interest.
\\

\noindent
\textbf{Ethics approval.} Not applicable.
\\

\end{document}